\documentclass[12pt]{article}
 
\usepackage[margin=1in]{geometry} 
\usepackage{amsmath,amsthm,amssymb,scrextend}
\usepackage{fancyhdr}
\usepackage{tikz}
\usepackage{pgfplots}
\usepackage{amsmath}
\usepackage[mathscr]{euscript}
 \let\mathscr\relax
\usepackage[scr]{rsfso}
\usepackage{amsthm}
\usepackage{amssymb}
\usepackage{multicol}
\usepackage{bbm}
\usepackage{mathtools}
\usepackage{cite}
\usepackage[colorlinks=true, pdfstartview=FitV, linkcolor=blue,
citecolor=blue, urlcolor=blue]{hyperref}

\theoremstyle{definition}

\newtheorem{assumption}{Assumption}

\newtheorem{remark}{Remark}

\newtheorem{theorem}{Theorem}
\newtheorem{lemma}{Lemma}

\begin{document}
 
\lhead{}
\chead{}
\rhead{}

\centerline{
 \bf When an Approximate Model Suffices for Optimal Control
}
\vspace{4px}

\begin{tabular*}{\textwidth}{l@{\extracolsep{\fill}}r}
	\text{\centerline{Andreas A. Malikopoulos}}\vspace{5pt}\\ 
	\centerline{Cornell University}\\
		\centerline{amaliko@cornell.edu}
\end{tabular*}

\date{March 2024}

\begin{abstract}
	
	In this paper, we develop an optimal control framework for dynamical systems when only an approximate
	model of the underlying plant is available. We consider a setting in which the
	control strategy is synthesized using a model-based optimal control problem that
	includes a penalty term capturing deviation from the plant trajectory, while the
	same control input is applied to both the model and the actual system. For a
	general class of optimal control problems, we establish conditions under
	which the control minimizing the model-based Hamiltonian coincides with the
	plant-optimal control, despite mismatch between the model and the true dynamics.
	We further specialize these results to problems with quadratic control effort,
	where explicit and easily verifiable sufficient conditions guarantee equivalence
	and uniqueness of the resulting optimal control.
	These results show that accurate control synthesis does not require an exact
	model of the underlying system, but rather alignment of the optimality conditions
	that govern control selection. From a learning perspective, this suggests that
	data-driven efforts can focus on identifying regimes in which model-based and
	plant-based Hamiltonian minimizers coincide, thereby providing a theoretical
	basis for robust model-based decision making and the effective use of digital
	twins under modeling error. We provide examples to illustrate the theoretical
	findings and demonstrate equivalence of the resulting control trajectories even
	in the presence of significant model mismatch.

\end{abstract}

\section{Introduction} \label{sec:introduction}

In many control applications, an accurate model of the system dynamics is either unavailable or prohibitively expensive to obtain. Nevertheless, approximate models, often referred to as digital twins, are typically used to derive control strategies imposing limitations in optimality and implications on the robust operation of the system. This raises the following fundamental question: ``Under what conditions is an optimal control strategy derived from an approximate model guaranteed to remain optimal for the actual system?" Addressing this question requires enhancing our understanding of the structural properties of optimal control problems that render them insensitive to certain forms of model mismatch. This paper is motivated by the need for a principled explanation of when and why approximate models are sufficient for deriving optimal control strategies.

\subsection{Literature Review}

A large body of work in control theory has addressed uncertainty through adaptive control 
\cite{strm1989AdaptiveC,Ioannou2012RobustAC}. 
These approaches have been particularly effective for regulation and tracking problems, where stability and performance guarantees can be preserved despite parametric uncertainty and unmodeled dynamics 
\cite{Dydek2013AdaptiveCO,Leman2009L1AC}. 
Adaptive control frameworks typically rely on explicit model structures and update laws, enabling safety guarantees but often limiting flexibility in complex or data-rich environments.

Reinforcement learning (RL) has emerged as a data-driven paradigm for control and decision making, drawing on foundations in dynamic programming and approximate optimal control 
\cite{Bertsekas1996,Sutton1998a,Sutton:1992ub}. 
RL methods aim to synthesize control strategies directly from interaction with the system, without requiring an explicit model, and have been applied to a wide range of dynamical systems. 
However, classical RL formulations often lack safety guarantees during learning, motivating significant recent efforts on integrating learning with control-theoretic tools.

Several research directions have focused on safe learning, combining ideas from reachability analysis, robustness, and probabilistic modeling. 
Representative examples include approaches that merge robust invariant sets with Bayesian learning mechanisms 
\cite{Fisac:2019wb}, as well as methods that model unknown dynamics using Gaussian processes to iteratively approximate maximal safe sets from data 
\cite{Akametalu:2014th}. 
These approaches explicitly trade off exploration and safety, but typically rely on conservative assumptions or computationally demanding updates.

Iterative learning control (ILC) constitutes another important class of methods for systems operating under repetitive tasks 
\cite{Armstrong:2021vw,Liu:2014wb}. 
ILC schemes exploit repeated executions to refine input trajectories so that system outputs converge to a desired reference. 
Related efforts combine ILC with extremum seeking techniques to enable online performance optimization without explicit system identification 
\cite{Khong:2016ug,Khong:2016ws}. 
In networked and cyber-physical systems, where accurate models may only be available through data samples, learning-based strategies have also been developed for allocating limited resources, such as communication or power budgets, to optimize overall control performance 
\cite{Gatsis:2021tt}.

A significant amount of work has addressed sequential decision-making under uncertainty using both direct and indirect learning-based approaches. 
These include learning-based model predictive control frameworks 
\cite{Rosolia:2018wv,Zhang:2020wf} and their application to high-performance autonomous driving tasks 
\cite{Rosolia:2020uo}. 
Learning has also been applied in real time to capture and adapt to human behavior, for example, in vehicle powertrain control tailored to driver preferences 
\cite{Malikopoulos2009,Malikopoulos2009b,Malikopoulos2010a}. 
Additional applications include learning-based planning for autonomous vehicles 
\cite{You:2019va}, traffic control and coordination in simulation environments 
\cite{Vinitsky:2018vx}, transfer of learned strategies from simulation to scaled physical testbeds 
\cite{jang2019simulation,chalaki2020ICCA}, decentralized learning in stochastic games 
\cite{Arslan:2017vo}, learning for social routing and congestion games 
\cite{Krichene:2015vx}, and security-oriented learning methods for cyber-physical systems subject to replay attacks 
\cite{Sahoo:2020tx}.

Many learning-based control approaches rely on an idealized or nominal model—such as an assumed transition kernel—to compute strategies that are subsequently deployed on the actual system. 
The sensitivity of optimal control strategies to model mismatch, and the regularity properties of optimal solutions with respect to perturbations in the transition kernel, have been studied in \cite{Kara:2018vu}, providing insight into robustness limits of model-based planning. 
Related efforts have addressed approximate planning and learning in partially observed systems through the notion of an information state 
\cite{Subramanian2020ApproximateIS}, as well as the construction of approximate information states derived from data and coupled with approximate dynamic programming schemes 
\cite{Subramanian2019ApproximateIS}. 
These approaches offer a principled pathway for RL in partially observed settings. 
More recent work has also explored hybrid architectures that combine model reference adaptive control with reinforcement learning to generate online strategies 
\cite{Guha2021OnlinePF}.

In a closely related line of work, we developed a theoretical foundation for integrating learning and optimal control in systems with unknown dynamics. We developed a separation-based framework and shown that optimal control strategies can be synthesized offline using an available model and implemented online once an appropriately defined information state is learned from data \cite{Malikopoulos2022a,Malikopoulos2024} and demonstrated it in a lin-
ear–quadratic regulator problem \cite{kounatidis2025combined}.

In this paper, we explicitly accounting for model mismatch through penalized model-based optimal control formulations that capture deviation from the actual system. These results demonstrate that optimal control can often be achieved without exact model identification, provided that learning preserves the structural properties underlying equivalence between model-based and plant-based decision making.

Detailed overviews of RL formulations, algorithms, and applications can be found in the survey papers  \cite{Kiumarsi:2018tq,Recht2018ATO}.

\subsection{Contributions of This Paper}

Building on the results reported in \cite{Malikopoulos2022a,Malikopoulos2024}, in this paper, we focus on identifying conditions under which optimal control strategies derived from an approximate model remain optimal for the actual system, even in the presence of model mismatch.
In particular, in this paper, we establish an optimal control framework that explains when and why a digital twin or approximate model can be used to derive an optimal control strategy for an actual system whose dynamics are unknown or mismatched. The key contribution is a set of structural conditions under which the constrained Hamiltonian minimizers of the model-based and plant-based optimal control problems coincide, implying equivalence of the resulting optimal control trajectories despite differences in system dynamics. We introduce a penalized model-based cost and analyze the associated Hamiltonian systems to demonstrate that exact model fidelity is not a prerequisite for optimal decision making. The results provide a principled explanation for the success of model-based and learning-enabled control architectures in practice, clarify the role of control constraints and cost structure in enforcing strategy equivalence, and offer a theoretical foundation for using digital twins as reliable surrogates for control synthesis even in the presence of modeling error.

\subsection{Comparison With Related Work}

The distinguishing feature of this paper relative to existing learning-based and model-based control approaches is that it provides an optimal control characterization of when a control strategy synthesized using an approximate model is also optimal for the actual system, despite model mismatch. In contrast to adaptive control and RL methods, which typically rely on asymptotic convergence, online identification, or robustness margins to mitigate modeling errors, our approach explicitly embeds model mismatch into the cost function through a time-varying penalty term and studies its impact at the level of the Hamiltonian minimization. This allows us to establish conditions under which the optimal control trajectories of the model-based and plant-based problems coincide exactly over a finite horizon.

Existing work on robust and learning-based control often focuses on performance degradation bounds, stability guarantees, or safe exploration under uncertainty, without addressing whether the control strategies derived from a nominal or learned model are structurally identical to the one obtained from the true system. By contrast, the results in this paper show that, under mild regularity and convexity conditions, the presence of a mismatch penalty does not alter the minimizer of the Hamiltonian, thereby preserving optimality of the control input itself. This perspective differs fundamentally from approaches based on robust optimization, model predictive control, or approximate dynamic programming, where mismatch typically leads to conservative or modified strategies.

Finally, while recent work has explored digital twins and learned models as surrogates for control design, such approaches are often justified empirically or algorithmically. The contribution of this paper is to provide a first-principles, continuous-time optimal control explanation of when and why a digital twin can be safely used to synthesize an optimal control strategy for the actual system. As such, the paper complements existing learning-based and adaptive methods by offering structural guarantees that are independent of the learning algorithm and rely solely on properties of the underlying optimization problem.

\subsection{Organization of the Paper}

The remainder of the paper is organized as follows. 
In Section~2, we introduce the problem formulation, including the dynamics of the actual system and the available model, the admissible control set, and the original and penalized optimal control problems. 
In Section~3, we develop a Hamiltonian-based analysis and establish key structural properties of the model-based and plant-based optimal control formulations. 
In Section~4, we present the main equivalence results and characterize conditions under which the constrained Hamiltonian minimizers and, consequently, the optimal control trajectories, coincide despite model mismatch. 
In Section~5, we provide illustrative numerical examples that validate the theoretical results and highlight the role of control constraints and cost structure. 
Finally, in Section~6, we draw concluding remarks and discuss potential directions for future research.

\section{Problem Formulation} \label{sec:problem_formulation}

We consider a finite-horizon optimal control problem for a continuous-time dynamical system whose exact dynamics are unknown. Although the state of the actual system is fully observed, the lack of an explicit model prevents direct offline computation of an optimal control strategy. To address this challenge, we assume the availability of a known model (digital twin) of the system dynamics. This model is used in conjunction with a suitably modified cost functional to derive a control strategy that, under appropriate conditions, coincides with the optimal control for the actual system.

In our exposition, we restrict attention to deterministic systems, a single control input, and full state observation, which allows us to isolate the core ideas without introducing additional technical complexity.

\subsection{Modeling framework}

We consider that the system with a state space $\mathbb{R}^n$, and control space is $\mathbb{R}^m$, $n,m\in\mathbb{N},$ evolves over a fixed finite time horizon $T>0$. Let $\mathcal U \subset \mathbb{R}^m$ be a given set of admissible control values.
An admissible control is any measurable function
\[
u:[0,T]\to\mathcal U.
\]

The actual system (plant) evolves according to the dynamics
\begin{equation}
\dot{\hat x}(t) = \hat f\big(t,\hat x(t),u(t)\big), 
\qquad \hat x(0)=x_0,
\label{eq:plant}
\end{equation}
where $\hat x(t)\in\mathbb{R}^n$ is the plant state and $\hat f:[0,T]\times\mathbb{R}^n\times\mathbb{R}^m\to\mathbb{R}^n$ is an \emph{unknown} dynamics map.

We consider that the plant state $\hat x(t)$ is fully observed for all $t\in[0,T]$ along with standard regularity conditions (e.g., Carath\'eodory conditions) such that for any admissible control $u(\cdot)\in\mathcal U$, the system \eqref{eq:plant} admits a unique absolutely continuous solution.

In addition to the plant (actual system), we consider that we have access to a known model of the system dynamics, given by
\begin{equation}
\dot x(t) = f\big(t,x(t),u(t)\big),
\qquad x(0)=x_0,
\label{eq:model}
\end{equation}
where $x(t)\in\mathbb{R}^n$ is the model state and $f:[0,T]\times\mathbb{R}^n\times\mathbb{R}^m\to\mathbb{R}^n$ is known.

The model and the plant share the same initial condition and are driven by the same control input $u(\cdot)$. The model state $x(t)$ is available at all times.

\subsection{Original optimal control problem for the actual system}

Let $\ell:[0,T]\times\mathbb{R}^n\times\mathbb{R}^m\to\mathbb{R}$ be a running cost and let $\phi:\mathbb{R}^n\to\mathbb{R}$ be a terminal cost. Given an admissible control $u(\cdot)\in\mathcal U$, the performance of the actual system is evaluated by
\begin{equation}
J_{\mathrm{act}}\big(u(\cdot)\big)
=
\int_{0}^{T}\ell\big(t,\hat x(t),u(t)\big)\,dt
+
\phi\big(\hat x(T)\big),
\label{eq:actual_cost}
\end{equation}
where $\hat x(\cdot)$ is the plant trajectory generated by \eqref{eq:plant}.
The objective is to solve the following problem:
\begin{quote}
\textbf{Problem P1 (Plant-optimal control).}  
Minimize $J_{\mathrm{act}}\big(u(\cdot)\big)$ over $u(\cdot)\in\mathcal U$, subject to the plant dynamics \eqref{eq:plant}.
\end{quote}

Because the plant dynamics $\hat f$ are unknown, Problem~P1 cannot be solved directly using standard optimal control techniques.

\subsection{Model-based surrogate problem with penalized cost}

To overcome the lack of knowledge of $\hat f$, we construct a surrogate optimal control problem based on the known model dynamics \eqref{eq:model}. The key idea is to augment the running cost with a penalty term that quantifies the discrepancy between the model state and the observed plant state.

Let $\beta:[0,T]\to[0,\infty)$ be a given time-varying weighting function. For any admissible control $u(\cdot)\in\mathcal U$, define the model-based penalized cost
\begin{equation}
J_{\mathrm{mod}}\big(u(\cdot);\hat x(\cdot)\big)
=
\int_{0}^{T}
\Big(
\ell\big(t,x(t),u(t)\big)
+
\beta(t)\,\|x(t)-\hat x(t)\|^2
\Big)\,dt
+
\phi\big(x(T)\big),
\label{eq:model_cost}
\end{equation}
where $x(\cdot)$ is generated by the model dynamics \eqref{eq:model} under the control $u(\cdot)$, and $\hat x(\cdot)$ is the observed plant trajectory generated by \eqref{eq:plant}.

Then, we consider the following problem:
\begin{quote}
\textbf{Problem P2 (Model-based penalized control).}  
Minimize $J_{\mathrm{mod}}\big(u(\cdot);\hat x(\cdot)\big)$ over $u(\cdot)\in\mathcal U$, subject to the model dynamics \eqref{eq:model}.
\end{quote}

The penalty term $\beta(t)\|x(t)-\hat x(t)\|^2$ encourages alignment between the model and the plant trajectories while preserving the original performance objective. Although the penalty term influences the state and costate evolution of the model-based problem, it does not enter explicitly in the pointwise minimization of the Hamiltonian with respect to the control input, a fact that will play a central role in the equivalence results developed later.

The objective of this paper is to identify conditions under which the optimal control derived from Problem~P2 coincides with the optimal control of the original plant problem~P1, despite the mismatch between the true dynamics $\hat f$ and the available model $f$.

\section{Hamiltonian Analysis and Necessary Conditions} \label{sec:HANC}

In this section, we derive the Hamiltonian formulation and necessary optimality conditions for the two problems introduced in Section~2, i.e., the original optimal control problem for the actual system (Problem~P1) and the model-based penalized problem (Problem~P2). Throughout our exposition, we consider that the regularity conditions required for the application of Pontryagin’s Minimum Principle (PMP) are satisfied.

\subsection{Hamiltonian for the actual system}

Recall that the actual system evolves according to
\[
\dot{\hat x}(t)=\hat f\big(t,\hat x(t),u(t)\big), \qquad \hat x(0)=x_0,
\]
and that the corresponding cost functional is
\[
J_{\mathrm{act}}\big(u(\cdot)\big)
=
\int_{0}^{T}\ell\big(t,\hat x(t),u(t)\big)\,dt
+
\phi\big(\hat x(T)\big).
\]

We define the Hamiltonian associated with Problem~P1 as
\begin{equation}
\hat H\big(t,\hat x,u,\hat\lambda\big)
=
\ell\big(t,\hat x,u\big)
+
\hat\lambda^{\top}\hat f\big(t,\hat x,u\big),
\label{eq:plant_hamiltonian}
\end{equation}
where $\hat\lambda(t)\in\mathbb{R}^n$ is the costate associated with the plant state $\hat x(t)$.

Based on PMP, if $u^*(\cdot)\in\mathcal U$ is an optimal control for Problem~P1 with corresponding state trajectory $\hat x^*(\cdot)$, then there exists a continuous costate trajectory $\hat\lambda^*(\cdot)$ such that, for almost every $t\in[0,T]$,
\begin{align}
\dot{\hat x}^*(t) &= \hat f\big(t,\hat x^*(t),u^*(t)\big), \label{eq:plant_state}\\
\dot{\hat\lambda}^*(t)
&=
-\nabla_{\hat x}\ell\big(t,\hat x^*(t),u^*(t)\big)
-\big(\nabla_{\hat x}\hat f\big(t,\hat x^*(t),u^*(t)\big)\big)^{\top}\hat\lambda^*(t),
\label{eq:plant_costate}
\end{align}
with terminal condition
\begin{equation}
\hat\lambda^*(T)=\nabla_{\hat x}\phi\big(\hat x^*(T)\big).
\label{eq:plant_terminal}
\end{equation}

Moreover, the optimal control $u^*(t)$ satisfies the pointwise constrained minimization condition
\begin{equation}
u^*(t)\in \arg\min_{u\in\mathcal U}
H\bigl(t,x^*(t),u,\lambda^*(t)\bigr),
\qquad \text{for a.e. } t\in[0,T].
\label{eq:plant_minimization}
\end{equation}

\subsection{Hamiltonian for the model-based penalized problem}

We now consider the model-based surrogate problem (Problem~P2), where we use the model given by
\[
\dot x(t)=f\big(t,x(t),u(t)\big), \qquad x(0)=x_0,
\]
and the penalized cost functional 
\[
J_{\mathrm{mod}}\big(u(\cdot);\hat x(\cdot)\big)
=
\int_{0}^{T}
\Big(
\ell\big(t,x(t),u(t)\big)
+
\beta(t)\,\|x(t)-\hat x(t)\|^2
\Big)\,dt
+
\phi\big(x(T)\big).
\]

The Hamiltonian associated with Problem~P2 is 
\begin{equation}
H\big(t,x,\hat x,u,\lambda\big)
=
\ell\big(t,x,u\big)
+
\beta(t)\,\|x-\hat x\|^2
+
\lambda^{\top} f\big(t,x,u\big),
\label{eq:model_hamiltonian}
\end{equation}
where $\lambda(t)\in\mathbb{R}^n$ is the costate associated with the model state $x(t)$, and $\hat x(t)$ enters as a known exogenous signal.

If $u^\circ(\cdot)\in\mathcal U$ is an optimal control for Problem~P2 with corresponding state trajectory $x^\circ(\cdot)$, then there exists an absolutely continuous costate trajectory $\lambda^\circ(\cdot)$ such that, for almost every $t\in[0,T]$,
\begin{align}
\dot x^\circ(t) &= f\big(t,x^\circ(t),u^\circ(t)\big), \label{eq:model_state}\\
\dot\lambda^\circ(t)
&=
-\nabla_x\ell\big(t,x^\circ(t),u^\circ(t)\big)
-2\beta(t)\big(x^\circ(t)-\hat x(t)\big)
-\big(\nabla_x f\big(t,x^\circ(t),u^\circ(t)\big)\big)^{\top}\lambda^\circ(t),
\label{eq:model_costate}
\end{align}
with terminal condition
\begin{equation}
\lambda^\circ(T)=\nabla_x\phi\big(x^\circ(T)\big).
\label{eq:model_terminal}
\end{equation}

The optimal control $u^\circ(t)$ satisfies the pointwise constrained minimization condition
\begin{equation}
u^\circ(t)\in
\arg\min_{u\in\mathcal U}
H\big(t,x^\circ(t),\hat x(t),u,\lambda^\circ(t)\big),
\qquad \text{for a.e. } t\in[0,T].
\label{eq:model_minimization}
\end{equation}

\subsection{Notes}
\begin{enumerate}
\item The Hamiltonians \eqref{eq:plant_hamiltonian} and \eqref{eq:model_hamiltonian} differ only in the dynamics and in the presence of the penalty term $\beta(t)\|x-\hat x\|^2$. This penalty term does not depend explicitly on the control variable at a fixed time and therefore does not appear in the pointwise optimality condition with respect to $u$. Its influence on the optimal control is indirect, through its effect on the state and costate trajectories of the model-based problem.

\item In practice, the proposed framework is envisioned to operate by running the
actual system and the available model simultaneously under the same control
input. At each time instant, the plant state $\hat x(t)$ is measured online and
treated as an exogenous signal within the model-based optimal control problem.
In particular, the measured plant trajectory enters the model Hamiltonian
through the penalty term, without affecting the pointwise minimization with
respect to the control input. This allows the model to be used for control
synthesis even when its dynamics do not accurately predict the plant, as long as
the resulting optimality conditions governing control selection remain aligned.
As a result, the framework avoids explicit model identification and instead
leverages real-time plant information to preserve equivalence between model-based
and plant-based optimal control strategies.

\end{enumerate}

\subsection{Constrained Hamiltonian Minimization--Existence and Uniqueness}

Next, we provide conditions under which the pointwise Hamiltonian minimization problems that arise in \eqref{eq:plant_minimization} and \eqref{eq:model_minimization} admit minimizers.

\begin{assumption}
\label{ass:U_relaxed}
The admissible control set $\mathcal U\subset\mathbb R^m$ is nonempty, closed, and convex (not necessarily bounded).
\end{assumption}

Assumption~\ref{ass:U_relaxed} guarantees well-posedness of the pointwise Hamiltonian minimization problem and enables the use of projection-based optimality conditions. Intuitively, this assumption ensures that there is always at least one feasible control choice and that small perturbations of admissible controls remain admissible, so optimal controls can be characterized geometrically without pathological boundary behavior.

\begin{assumption}
\label{ass:convex_coercive}
For almost every $t\in[0,T]$ and for all relevant $(\hat x,\hat\lambda)$ and $(x,\hat x,\lambda)$, the maps
\[
u\mapsto \hat H(t,\hat x,u,\hat\lambda)
\quad\text{and}\quad
u\mapsto H(t,x,\hat x,u,\lambda)
\]
are proper, lower semicontinuous, and convex on $\mathcal U$.
Moreover, they are \emph{coercive on $\mathcal U$}, i.e.,
\[
\|u\|\to\infty,\ u\in\mathcal U
\quad\Longrightarrow\quad
\hat H(t,\hat x,u,\hat\lambda)\to+\infty
\ \text{and}\
H(t,x,\hat x,u,\lambda)\to+\infty.
\]
\end{assumption}

Assumption~\ref{ass:convex_coercive} guarantees existence and attainability of minimizers for the pointwise Hamiltonian minimization problem by enforcing proper lower semicontinuity and coerciveness with respect to the control input.
In simple terms, this assumption prevents minimizing sequences from escaping to infinity or failing to converge, ensuring that an optimal control can actually be realized rather than existing only as an infimum.

\begin{theorem}[Existence and uniqueness]
\label{thm:exist_unique_minimizers_relaxed}
Suppose Assumptions~\ref{ass:U_relaxed}--\ref{ass:convex_coercive} hold. Then, for almost every $t\in[0,T]$, the sets of minimizers
\[
\arg\min_{u\in\mathcal U}\hat H(t,\hat x,u,\hat\lambda)
\quad\text{and}\quad
\arg\min_{u\in\mathcal U} H(t,x,\hat x,u,\lambda)
\]
are nonempty, closed, and convex.

If, in addition, for almost every $t$ the Hamiltonians are \emph{strictly convex} in $u$ on $\mathcal U$ (e.g., $\alpha$-strongly convex), then these minimizers are unique almost everywhere.
\end{theorem}

\begin{proof}
We fix $t\in[0,T]$ for which the stated convexity and coercivity properties hold (this is assumed to be the case for almost every $t$). For clarity of exposition, we prove the claims for the model-based Hamiltonian only as the argument for the plant Hamiltonian is identical.

\medskip\noindent
We define the extended-value function
\[
\Psi(u) := H(t,x,\hat x,u,\lambda) + \iota_{\mathcal U}(u),
\]
where $\iota_{\mathcal U}$ is the indicator of $\mathcal U$, i.e.,
\[
\iota_{\mathcal U}(u)=
\begin{cases}
	0, & u\in\mathcal U,\\
	+\infty, & u\notin\mathcal U.
\end{cases}
\]
Then minimizing $H(t,x,\hat x,\cdot,\lambda)$ over $\mathcal U$ is equivalent to minimizing $\Psi$ over $\mathbb R^m$:
\[
\arg\min_{u\in\mathcal U} H(t,x,\hat x,u,\lambda)
=
\arg\min_{u\in\mathbb R^m}\Psi(u).
\]
By Assumption~\ref{ass:convex_coercive}, $u\mapsto H(t,x,\hat x,u,\lambda)$ is proper, lower semicontinuous, and convex on $\mathcal U$, and by Assumption~\ref{ass:U_relaxed} the set $\mathcal U$ is nonempty, closed, and convex. Hence $\Psi$ is a proper, lower semicontinuous, convex function on $\mathbb R^m$.

\medskip\noindent
Next, we show that $\Psi$ attains its minimum. By coercivity of $H$ on $\mathcal U$, we have
\[
\|u\|\to\infty,\ u\in\mathcal U \quad\Longrightarrow\quad H(t,x,\hat x,u,\lambda)\to +\infty,
\]
and thus also
\[
\|u\|\to\infty \quad\Longrightarrow\quad \Psi(u)\to +\infty,
\]
because $\Psi(u)=+\infty$ outside $\mathcal U$.

Let $m^\star:=\inf_{u\in\mathbb R^m}\Psi(u)=\inf_{u\in\mathcal U}H(t,x,\hat x,u,\lambda)$, and choose a minimizing sequence $\{u_k\}_{k\ge 1}\subset\mathbb R^m$ such that $\Psi(u_k)\downarrow m^\star$. Since $\Psi(u_k)<+\infty$ for all large $k$, we have $u_k\in\mathcal U$ for all large $k$. Coercivity implies that the sublevel set
\[
\mathcal L_c := \{u\in\mathbb R^m: \Psi(u)\le c\}
\]
is bounded for any finite $c$. In particular, for $c=\Psi(u_1)$, all sufficiently large $k$ satisfy $u_k\in\mathcal L_c$, hence $\{u_k\}$ is bounded.

\medskip\noindent
Therefore, by Bolzano--Weierstrass, there exists a convergent subsequence (not relabeled) such that $u_k\to \bar u$. Since $\mathcal U$ is closed and $u_k\in\mathcal U$ for all large $k$, it follows that $\bar u\in\mathcal U$.

By the lower semicontinuity of $\Psi$,
\[
\Psi(\bar u)\le \liminf_{k\to\infty}\Psi(u_k)=m^\star,
\]
which implies $\Psi(\bar u)=m^\star$. Hence $\bar u$ is a minimizer, and the argmin set is nonempty, i.e.,
\[
\arg\min_{u\in\mathcal U} H(t,x,\hat x,u,\lambda)\neq\emptyset.
\]

To show closedness of the argmin set, let $\mathcal M := \arg\min_{u\in\mathcal U} H(t,x,\hat x,u,\lambda)$. Since $\Psi$ is lower semicontinuous, $\mathcal M$ is closed, i.e., if $u_k\in\mathcal M$ and $u_k\to u$, then
\[
\Psi(u)\le \liminf_{k\to\infty}\Psi(u_k)=m^\star,
\]
so $u\in\mathcal M$.

To show convexity of $\mathcal M$, let $u_1,u_2\in\mathcal M$ and $\theta\in[0,1]$. Using convexity of $\Psi$,
\[
\Psi(\theta u_1+(1-\theta)u_2)
\le
\theta \Psi(u_1)+(1-\theta)\Psi(u_2)
=
\theta m^\star+(1-\theta)m^\star
=
m^\star,
\]
hence $\theta u_1+(1-\theta)u_2\in\mathcal M$. Thus $\mathcal M$ is convex.

To show uniqueness under strict convexity, we assume that $u\mapsto H(t,x,\hat x,u,\lambda)$ is strictly convex on $\mathcal U$. Suppose, for contradiction, that there exist two distinct minimizers $u_1\neq u_2$ in $\mathcal M$. For any $\theta\in(0,1)$, strict convexity implies
\[
H\big(t,x,\hat x,\theta u_1+(1-\theta)u_2,\lambda\big)
<
\theta H(t,x,\hat x,u_1,\lambda)+(1-\theta)H(t,x,\hat x,u_2,\lambda)
=
m^\star,
\]
which contradicts the definition of $m^\star$ as the infimum. Hence the minimizer is unique.

The proof for $\hat H(t,\hat x,\cdot,\hat\lambda)$ is identical under the corresponding convexity and coercivity assumptions, yielding nonemptiness, closedness, and convexity of its argmin set, and uniqueness under strict convexity.
This completes the proof.
\end{proof}

\section{Equivalence of Optimal Controls}

All equivalence results in this section are pointwise in time and rely on the structure of the instantaneous Hamiltonian minimization problems induced by the two optimal control formulations.
We provide two complementary equivalence results. The first is stated in a convex-analysis form (subdifferentials and normal cones) and accommodates nonsmooth costs, unbounded control sets, and nonlinear dynamics, provided the pointwise Hamiltonian minimization problems are convex. The second result specializes to a commonly used structural setting (quadratic control effort and mild growth conditions), which yields simple, verifiable conditions for existence, uniqueness, and equivalence.

\subsection{Preliminaries from convex analysis}

Let $\mathcal U\subset\mathbb R^m$ be nonempty, closed, and convex. The \emph{normal cone} to $\mathcal U$ at $u\in\mathcal U$ is defined by
\[
N_{\mathcal U}(u) := \{ \eta\in\mathbb R^m : \langle \eta, v-u\rangle \le 0,\ \forall v\in\mathcal U\}.
\]
For a proper, lower semicontinuous, convex function $\psi:\mathbb R^m\to\mathbb R\cup\{+\infty\}$, the convex subdifferential at $u$ is denoted by $\partial \psi(u)$.

We will use the standard fact that $u^*\in\arg\min_{u\in\mathcal U}\psi(u)$ if and only if
\[
0 \in \partial \psi(u^*) + N_{\mathcal U}(u^*).
\]

\begin{remark}[First-order optimality in variational inequality form]
\label{rem:VI_relaxed}
Under Assumption~\ref{ass:U_relaxed} and convexity of $u\mapsto H(t,\cdot)$, $u^*(t)$ minimizes $H(t,\cdot)$ over $\mathcal U$ if and only if
\[
0\in \partial_u H(t,u^*(t)) + N_{\mathcal U}(u^*(t)),
\]
where $\partial_u$ denotes the convex subdifferential and $N_{\mathcal U}$ is the normal cone to $\mathcal U$.
If $H$ is differentiable in $u$, this reduces to the variational inequality
\[
\langle \nabla_u H(t,u^*(t)),\, v-u^*(t)\rangle\ge 0,\qquad \forall v\in\mathcal U.
\]
\end{remark}

\subsection{General equivalence results}

To this end, for notational simplicity, we suppress the explicit dependence on time \(t\) when no ambiguity arises.

Recall the Hamiltonians from Section~3:
\begin{align*}
\hat H(t,\hat x,u,\hat\lambda) &= \ell(t,\hat x,u) + \hat\lambda^\top \hat f(t,\hat x,u),\\
H(t,x,\hat x,u,\lambda) &= \ell(t,x,u) + \beta(t)\|x-\hat x\|^2 + \lambda^\top f(t,x,u).
\end{align*}

\begin{theorem}
\label{thm:equiv_optionA}
Suppose Assumptions~\ref{ass:U_relaxed}--\ref{ass:convex_coercive} hold. Fix any time $t$ for which the pointwise Hamiltonian
minimizers exist, and suppress dependence on $(t,\cdot)$ for readability.
Suppose there exists $\bar u\in\mathcal U$ such that
\begin{equation}
	\partial_u \hat H(\hat x,\bar u,\hat\lambda)=\partial_u H(x,\hat x,\bar u,\lambda),
	\label{eq:grad_match}
\end{equation}
and, in addition,
\begin{equation}
	0\in \partial_u \hat H(\hat x,\bar u,\hat\lambda)+N_{\mathcal U}(\bar u).
	\label{eq:normal_cone_cond}
\end{equation}
Then $\bar u$ is a minimizer for both pointwise Hamiltonian problems, i.e.,
\[
\bar u\in\arg\min_{v\in\mathcal U}\hat H(\hat x,v,\hat\lambda)
\quad\text{and}\quad
\bar u\in\arg\min_{v\in\mathcal U}H(x,\hat x,v,\lambda).
\]
If, moreover, each Hamiltonian is strictly convex in $u$ on $\mathcal U$
(for a.e.\ $t$), then the minimizers are unique and must coincide.
\end{theorem}

\begin{proof}
For brevity, define the time-$t$ objective functions
\begin{align}
	\Psi_{\mathrm{act}}(u;t)
	&:=\hat H\big(t,\hat x^*(t),u,\hat\lambda^*(t)\big), 
	\label{eq:Psi_act_def}\\
	\Psi_{\mathrm{mod}}(u;t)
	&:= H\big(t,x^\circ(t),\hat x^*(t),u,\lambda^\circ(t)\big).
	\label{eq:Psi_mod_def}
\end{align}
By Assumptions~\ref{ass:U_relaxed}--\ref{ass:convex_coercive}, for almost every
$t\in[0,T]$ the pointwise minimization problems defining
$\Psi_{\mathrm{act}}(\cdot;t)$ and $\Psi_{\mathrm{mod}}(\cdot;t)$ are well posed
(i.e., admit minimizers) on the nonempty closed convex set $\mathcal U$. 
Moreover, both maps
$u\mapsto\Psi_{\mathrm{act}}(u;t)$ and $u\mapsto\Psi_{\mathrm{mod}}(u;t)$ are
convex and differentiable with respect to $u$ on $\mathcal U$.

Since $\Psi_{\mathrm{act}}(\cdot;t)$ is convex and differentiable, a control
$u=u^*(t)$ minimizes $\Psi_{\mathrm{act}}(\cdot;t)$ over $\mathcal U$ if and only
if it satisfies the variational inequality
\begin{equation}
	\Big\langle \nabla_u \Psi_{\mathrm{act}}(u^*(t);t),\, v-u^*(t)\Big\rangle \ge 0,
	\qquad \forall v\in\mathcal U.
	\label{eq:VI_act}
\end{equation}
Similarly, $u=u^\circ(t)$ minimizes $\Psi_{\mathrm{mod}}(\cdot;t)$ over
$\mathcal U$ if and only if
\begin{equation}
	\Big\langle \nabla_u \Psi_{\mathrm{mod}}(u^\circ(t);t),\, v-u^\circ(t)\Big\rangle
	\ge 0,
	\qquad \forall v\in\mathcal U.
	\label{eq:VI_mod}
\end{equation}
These conditions are the standard first-order optimality conditions for convex
minimization over a closed convex set.

By the definition of the Hamiltonians and since the penalty term
$\beta(t)\|x-\hat x\|^2$ is independent of $u$ for fixed
$(t,x,\hat x)$, we obtain
\begin{align}
	\nabla_u \Psi_{\mathrm{act}}(u;t)
	&= \nabla_u \ell\big(t,\hat x^*(t),u\big)
	+ \nabla_u\!\Big(\hat\lambda^{*\top}(t)\,
	\hat f\big(t,\hat x^*(t),u\big)\Big),
	\label{eq:grad_act}\\
	\nabla_u \Psi_{\mathrm{mod}}(u;t)
	&= \nabla_u \ell\big(t,x^\circ(t),u\big)
	+ \nabla_u\!\Big(\lambda^{\circ\top}(t)\,
	f\big(t,x^\circ(t),u\big)\Big).
	\label{eq:grad_mod}
\end{align}
By the hypothesis of Theorem~2, the gradients of the plant and model Hamiltonians
with respect to $u$ coincide when evaluated at $u=u^\circ(t)$, namely,
\begin{equation}
	\nabla_u \Psi_{\mathrm{act}}(u^\circ(t);t)
	=
	\nabla_u \Psi_{\mathrm{mod}}(u^\circ(t);t).
	\label{eq:grad_equality_at_u_circ}
\end{equation}

\medskip\noindent
Substituting the gradient-matching condition into the variational inequality
characterizing optimality of $u^\circ(t)$ for the model-based problem yields
\begin{equation}
	\Big\langle \nabla_u \Psi_{\mathrm{act}}(u^\circ(t);t),\,
	v-u^\circ(t)\Big\rangle \ge 0,
	\qquad \forall v\in\mathcal U.
	\label{eq:VI_act_at_u_circ}
\end{equation}
Since $\Psi_{\mathrm{act}}(\cdot;t)$ is convex and differentiable on
$\mathcal U$, the condition \eqref{eq:VI_act_at_u_circ} is equivalent to
\[
u^\circ(t)\in \arg\min_{u\in\mathcal U} \Psi_{\mathrm{act}}(u;t)
=
\arg\min_{u\in\mathcal U}
\hat H\big(t,\hat x^*(t),u,\hat\lambda^*(t)\big).
\]
Thus, the control $u^\circ(t)$ obtained from the model-based problem is also a
minimizer of the plant Hamiltonian $\hat H(\cdot)$ at time $t$.
Since the above argument holds for any time $t$ at which the PMP conditions
apply, the conclusion follows for almost every $t\in[0,T]$.
Consequently,
\[
u^\circ(t)\in\arg\min_{u\in\mathcal U}\Psi_{\mathrm{act}}(u;t)
\quad \text{for a.e. } t\in[0,T].
\]

If each Hamiltonian is strictly convex in $u$, the $\arg\min$ set is a singleton, and therefore
\[
u^\circ(t) = u^*(t) \quad \text{for a.e. } t\in[0,T].
\]
\end{proof}

\begin{remark}
\label{rem:interpretA}
Theorem~\ref{thm:equiv_optionA} is purely pointwise in time. In the PMP setting, if along two candidate extremals the subgradient/normal-cone conditions align almost everywhere in time, then the resulting control trajectories coincide almost everywhere (with uniqueness guaranteed under strict convexity).
\end{remark}

\subsection{Specialization to quadratic control effort}

While Theorem~\ref{thm:equiv_optionA} provides a general equivalence result under convexity, its hypotheses may be abstract to verify directly. Next, we specialize to a quadratic control effort under which the equivalence becomes explicit and easily verifiable.

\begin{assumption}
\label{ass:struct_B}
The admissible control set $\mathcal U\subset\mathbb R^m$ is nonempty, closed, and convex (possibly unbounded). The running cost has the form
\begin{equation}
	\label{eq:quad_u}
	\ell(t,z,u) = \ell_0(t,z) + \frac{1}{2} u^\top R(t) u,
\end{equation}
where $\ell_0(t,\cdot)$ is continuous in $z$ and $R(t)\in\mathbb R^{m\times m}$ satisfies
\[
R(t)\succeq r_{\min} I \quad \text{for some } r_{\min}>0\ \text{and all }t\in[0,T].
\]
\end{assumption}

Assumption~\ref{ass:struct_B} guarantees uniform strong convexity of the Hamiltonian with respect to the control input, ensuring existence and uniqueness of the pointwise optimal control and well-posedness of the minimization problem over the entire time horizon. Intuitively, this condition ensures that control effort is penalized in every direction at all times, so the optimal control cannot be flat, ill-defined, or sensitive to small perturbations.

\begin{assumption}
\label{ass:convex_growth_B}
For almost every $t$ and all relevant $(x,\hat x,\lambda,\hat\lambda)$, the maps
\[
u\mapsto \lambda^\top f(t,x,u),\qquad u\mapsto \hat\lambda^\top \hat f(t,\hat x,u)
\]
are convex on $\mathcal U$, and satisfy a linear growth bound, i.e., there exist locally bounded functions $c_f(t,x),c_{\hat f}(t,\hat x)\ge 0$ such that for all $u\in\mathcal U$,
\[
|\lambda^\top f(t,x,u)| \le c_f(t,x)\,(1+\|u\|),\qquad
|\hat\lambda^\top \hat f(t,\hat x,u)| \le c_{\hat f}(t,\hat x)\,(1+\|u\|).
\]
\end{assumption}

Assumption~\ref{ass:convex_growth_B} imposes a linear-growth condition on the control-dependent terms of the Hamiltonian, ensuring coercivity and preventing unbounded descent even when the admissible control set is unbounded. In simple terms, this condition guarantees that no term in the dynamics or cost can overpower the quadratic control penalty, so the optimization does not “prefer” arbitrarily large control actions.

\begin{lemma}[Existence and uniqueness under quadratic effort]
\label{lem:exist_unique_B}
Under Assumptions~\ref{ass:struct_B}--\ref{ass:convex_growth_B}, for almost every $t$ the pointwise minimization problems
\[
\arg\min_{u\in\mathcal U}\hat H(t,\hat x,u,\hat\lambda),\qquad
\arg\min_{u\in\mathcal U}H(t,x,\hat x,u,\lambda)
\]
admit unique minimizers.
\end{lemma}

\begin{proof}
Fix $t\in[0,T]$ such that Assumptions~\ref{ass:struct_B}--\ref{ass:convex_growth_B} hold (this is the case for almost every $t$). We prove the claim for the model-based Hamiltonian. The proof for the plant Hamiltonian is identical.

For fixed $(t,x,\hat x,\lambda)$, we define the function
\[
\Psi(u):= H(t,x,\hat x,u,\lambda)
= \ell_0(t,x) + \frac{1}{2}u^\top R(t)u + \beta(t)\|x-\hat x\|^2 + \lambda^\top f(t,x,u),
\qquad u\in\mathcal U.
\]
By Assumption~\ref{ass:convex_growth_B}, the map $u\mapsto \lambda^\top f(t,x,u)$ is convex on $\mathcal U$. Therefore $\Psi$ is convex on $\mathcal U$. Moreover, since $R(t)\succeq r_{\min}I$, the quadratic term is $r_{\min}$-strongly convex on $\mathbb R^m$, hence $\Psi$ is \emph{strongly convex} on $\mathcal U$ as the sum of a strongly convex function and a convex function.

\medskip\noindent
Let $r_{\max}(t):=\|R(t)\|$ and note that
\[
\frac{1}{2}u^\top R(t)u \;\ge\; \frac{r_{\min}}{2}\|u\|^2 \qquad \forall u\in\mathbb R^m.
\]
By Assumption~\ref{ass:convex_growth_B}, there exists $c_f(t,x)\ge 0$ such that
\[
|\lambda^\top f(t,x,u)| \le c_f(t,x)\,(1+\|u\|),\qquad \forall u\in\mathcal U.
\]
Hence, for all $u\in\mathcal U$,
\begin{align*}
	\Psi(u)
	&\ge
	\Big(\ell_0(t,x)+\beta(t)\|x-\hat x\|^2\Big)
	+\frac{r_{\min}}{2}\|u\|^2
	- c_f(t,x)\,(1+\|u\|)\\
	&=
	C(t,x,\hat x) + \frac{r_{\min}}{2}\|u\|^2 - c_f(t,x)\|u\| - c_f(t,x),
\end{align*}
where $C(t,x,\hat x):=\ell_0(t,x)+\beta(t)\|x-\hat x\|^2$ does not depend on $u$. The right-hand side is a quadratic function of $\|u\|$ with positive leading coefficient $\frac{r_{\min}}{2}$; therefore,
\[
\|u\|\to\infty,\ u\in\mathcal U \quad \Longrightarrow \quad \Psi(u)\to +\infty.
\]
Namely, $\Psi$ is coercive on $\mathcal U$.

Let $m^\star:=\inf_{u\in\mathcal U}\Psi(u)$ and let $\{u_k\}\subset\mathcal U$ be a minimizing sequence with $\Psi(u_k)\downarrow m^\star$. Coercivity implies that $\{u_k\}$ is bounded. Otherwise $\|u_k\|\to\infty$ along a subsequence would force $\Psi(u_k)\to+\infty$, contradicting $\Psi(u_k)\downarrow m^\star<+\infty$.

Since $\{u_k\}$ is bounded, there exists a subsequence (not relabeled) and $\bar u\in\mathbb R^m$ such that $u_k\to \bar u$. Because $\mathcal U$ is closed, $\bar u\in\mathcal U$. Finally, $\Psi$ is convex (hence continuous on the relative interior of $\mathcal U$) and, under the present assumptions, lower semicontinuous on $\mathcal U$; thus
\[
\Psi(\bar u)\le \liminf_{k\to\infty}\Psi(u_k)=m^\star,
\]
which yields $\Psi(\bar u)=m^\star$. Therefore, $\bar u$ is a minimizer and the argmin set is nonempty.

Because $\Psi$ is strongly convex on $\mathcal U$, it admits at most one minimizer on $\mathcal U$. Indeed, if $u_1\neq u_2$ were both minimizers, then for any $\theta\in(0,1)$ strong convexity would imply
\[
\Psi(\theta u_1+(1-\theta)u_2)
<
\theta \Psi(u_1) + (1-\theta)\Psi(u_2)
=
m^\star,
\]
a contradiction. Hence, the minimizer is unique.

For fixed $(t,\hat x,\hat\lambda)$ define
\[
\hat\Psi(u):=\hat H(t,\hat x,u,\hat\lambda)
= \ell_0(t,\hat x) + \frac{1}{2}u^\top R(t)u + \hat\lambda^\top \hat f(t,\hat x,u),
\qquad u\in\mathcal U.
\]
Assumption~\ref{ass:convex_growth_B} gives convexity of $u\mapsto \hat\lambda^\top \hat f(t,\hat x,u)$ on $\mathcal U$ and the same linear growth bound, which again yields coercivity and strong convexity of $\hat\Psi$. Therefore, $\hat\Psi$ admits a unique minimizer over $\mathcal U$.

Thus, for almost every $t\in[0,T]$, both pointwise minimization problems admit unique minimizers, and the proof is complete.
\end{proof}

\begin{theorem}[Structural equivalence of optimal controls]
\label{thm:equiv_optionB}
Suppose Assumptions~\ref{ass:struct_B}--\ref{ass:convex_growth_B} hold. Let $(\hat x^*(\cdot),\hat\lambda^*(\cdot),u^*(\cdot))$ satisfy the PMP conditions for the plant problem (Problem~P1), and let $(x^\circ(\cdot),\lambda^\circ(\cdot),u^\circ(\cdot))$ satisfy the PMP conditions for the model-based penalized problem (Problem~P2).

Suppose that, for almost every $t\in[0,T]$,
\begin{equation}
	\label{eq:struct_match_switching}
	\nabla_u \ell(t,\hat x^*(t),u) + \nabla_u\!\big(\hat\lambda^{*\top}(t)\hat f(t,\hat x^*(t),u)\big)
	=
	\nabla_u \ell(t,x^\circ(t),u) + \nabla_u\!\big(\lambda^{\circ\top}(t) f(t,x^\circ(t),u)\big)
	\quad \text{at } u=u^\circ(t),
\end{equation}
and that the state alignment holds:
\begin{equation}
	\label{eq:state_align_B}
	x^\circ(t)=\hat x^*(t)\quad\text{for a.e. }t\in[0,T].
\end{equation}
Then
\[
u^\circ(t)=u^*(t)\quad\text{for a.e. }t\in[0,T].
\]

A sufficient set of verifiable conditions implying \eqref{eq:struct_match_switching} is:
\begin{equation}
	\label{eq:sufficient_B}
	\nabla_u\hat f(t,\hat x^*(t),u)=\nabla_u f(t,\hat x^*(t),u)\ \text{for all }u\in\mathcal U,\ \text{a.e. }t,
	\qquad
	\lambda^\circ(t)=\hat\lambda^*(t)\ \text{a.e.}
\end{equation}
\end{theorem}

\begin{proof}
We prove that $u^\circ(t)=u^*(t)$ for almost every $t\in[0,T]$. The proof is pointwise in time and relies on (i) uniqueness of the pointwise Hamiltonian minimizers (Lemma~\ref{lem:exist_unique_B}) and (ii) the switching-gradient matching condition \eqref{eq:struct_match_switching}.

By Lemma~\ref{lem:exist_unique_B}, under Assumptions~\ref{ass:struct_B}--\ref{ass:convex_growth_B}, for almost every $t\in[0,T]$ the pointwise minimization problems
\[
\arg\min_{u\in\mathcal U}\hat H\big(t,\hat x^*(t),u,\hat\lambda^*(t)\big),
\qquad
\arg\min_{u\in\mathcal U} H\big(t,x^\circ(t),\hat x^*(t),u,\lambda^\circ(t)\big)
\]
admit \emph{unique} minimizers. We fix such a time $t$ and suppress the explicit dependence on $t$ in the notation.

We define the pointwise objective functions
\begin{align}
	\Psi_{\mathrm{act}}(u)
	&:=\hat H\big(\hat x^*,u,\hat\lambda^*\big),
	\label{eq:thm3_Psi_act}\\
	\Psi_{\mathrm{mod}}(u)
	&:= H\big(x^\circ,\hat x^*,u,\lambda^\circ\big).
	\label{eq:thm3_Psi_mod}
\end{align}
Then, by definition of the PMP minimization conditions,
\begin{equation}
	u^*=\arg\min_{u\in\mathcal U}\Psi_{\mathrm{act}}(u),
	\qquad
	u^\circ=\arg\min_{u\in\mathcal U}\Psi_{\mathrm{mod}}(u).
	\label{eq:thm3_argmins}
\end{equation}

\medskip\noindent
Under Assumptions~\ref{ass:struct_B}--\ref{ass:convex_growth_B}, both $\Psi_{\mathrm{act}}$ and $\Psi_{\mathrm{mod}}$ are convex and (by the standing smoothness conditions in the PMP setup) differentiable in $u$. Therefore, the unique minimizer $u^\circ$ of $\Psi_{\mathrm{mod}}$ satisfies the variational inequality
\begin{equation}
	\big\langle \nabla_u \Psi_{\mathrm{mod}}(u^\circ),\, v-u^\circ \big\rangle \ge 0,
	\qquad \forall v\in\mathcal U.
	\label{eq:thm3_VI_mod}
\end{equation}
Similarly, $u^*$ is characterized by the corresponding variational inequality for $\Psi_{\mathrm{act}}$:
\begin{equation}
	\big\langle \nabla_u \Psi_{\mathrm{act}}(u^*),\, v-u^* \big\rangle \ge 0,
	\qquad \forall v\in\mathcal U.
	\label{eq:thm3_VI_act}
\end{equation}

\medskip\noindent
By the definitions of the Hamiltonians,
\begin{align}
	\nabla_u \Psi_{\mathrm{act}}(u)
	&=
	\nabla_u \ell\big(t,\hat x^*(t),u\big)
	+
	\nabla_u\!\Big(\hat\lambda^{*\top}(t)\,\hat f\big(t,\hat x^*(t),u\big)\Big),
	\label{eq:thm3_grad_act}\\
	\nabla_u \Psi_{\mathrm{mod}}(u)
	&=
	\nabla_u \ell\big(t,x^\circ(t),u\big)
	+
	\nabla_u\!\Big(\lambda^{\circ\top}(t)\, f\big(t,x^\circ(t),u\big)\Big),
	\label{eq:thm3_grad_mod}
\end{align}
where we have used that the penalty term $\beta(t)\|x^\circ(t)-\hat x^*(t)\|^2$ does not depend explicitly on $u$ at fixed $(t,x^\circ(t),\hat x^*(t))$, and hence does not contribute to $\nabla_u \Psi_{\mathrm{mod}}$.

From the hypothesis,
\begin{equation}
	\nabla_u \Psi_{\mathrm{act}}(u)\Big|_{u=u^\circ}
	=
	\nabla_u \Psi_{\mathrm{mod}}(u)\Big|_{u=u^\circ}.
	\label{eq:thm3_grad_match}
\end{equation}
(Condition \eqref{eq:state_align_B} ensures that the state arguments appearing in the two gradients are evaluated consistently along the relevant trajectory.)

Substituting \eqref{eq:thm3_grad_match} into the variational inequality \eqref{eq:thm3_VI_mod} yields
\begin{equation}
	\big\langle \nabla_u \Psi_{\mathrm{act}}(u^\circ),\, v-u^\circ \big\rangle \ge 0,
	\qquad \forall v\in\mathcal U.
	\label{eq:thm3_VI_act_at_ucirc}
\end{equation}

\medskip\noindent
Since $\Psi_{\mathrm{act}}$ is convex and differentiable on the closed convex set $\mathcal U$, the variational inequality \eqref{eq:thm3_VI_act_at_ucirc} is \emph{equivalent} to the statement that $u^\circ$ is a minimizer of $\Psi_{\mathrm{act}}$ over $\mathcal U$, i.e.,
\[
u^\circ \in \arg\min_{u\in\mathcal U}\Psi_{\mathrm{act}}(u).
\]
But by Lemma~\ref{lem:exist_unique_B}, this argmin set is a singleton $\{u^*\}$. Therefore,
\[
u^\circ = u^*.
\]

The argument above holds for every $t$ at which the pointwise minimizers are unique and the matching condition \eqref{eq:struct_match_switching} holds. From the hypothesis, \eqref{eq:struct_match_switching} and \eqref{eq:state_align_B} hold for almost every $t$, and by Lemma~\ref{lem:exist_unique_B} uniqueness holds for almost every $t$. Hence,
\[
u^\circ(t)=u^*(t)\quad\text{for a.e. }t\in[0,T],
\]
which completes the proof.

\medskip\noindent
Finally, the sufficient conditions \eqref{eq:sufficient_B} imply \eqref{eq:struct_match_switching} by direct substitution into \eqref{eq:thm3_grad_act}--\eqref{eq:thm3_grad_mod}, since equality of $\nabla_u \hat f$ and $\nabla_u f$ (together with $\lambda^\circ=\hat\lambda^*$ and $x^\circ=\hat x^*$) yields equality of the Hamiltonian gradients at $u=u^\circ(t)$.
\end{proof}

\begin{remark}
The equivalence of the optimal control laws derived from the plant and the model-based problems follows from the fact that both Hamiltonians induce the same variational inequality with respect to the control input. In particular, although the state and costate trajectories generally differ due to model mismatch and the presence of the penalty term, the latter does not depend on the control input and therefore does not affect the subdifferential of the Hamiltonian with respect to the control. As a result, the normal-cone optimality condition associated with the admissible control set is identical for both problems, leading to the same pointwise optimal control.
\end{remark}

\subsection{Illustrative examples}

\subsubsection{Example 1}

This example illustrates the equivalence result of
Theorem~\ref{thm:equiv_optionA} in a simple setting where the running cost is
\emph{nonsmooth}, thereby necessitating the use of subgradients and normal cones
instead of classical gradients.

Consider a scalar control system with admissible control set
\[
\mathcal U = [-1,1],
\]
and a running cost given by
\[
\ell(u) = |u|,
\]
which is convex but not differentiable at $u=0$. For simplicity, suppose that
the control enters linearly in both the plant and the model dynamics, so that
the corresponding Hamiltonians can be written as
\[
\hat H(u) = |u| + \hat a\,u + \beta\,\|x-\hat x\|^2,
\qquad
H(u) = |u| + a\,u,
\]
where \(a,\hat a \in \mathbb{R}\) capture the effect of the costates and the
control-dependent dynamics, and $\beta\ge 0$ is the penalty weight associated
with the model-based formulation.

Note that the penalty term $\beta\,\|x-\hat x\|^2$ is independent of the control
input $u$ and therefore does not affect the pointwise minimization with respect
to $u$ or the associated subdifferential conditions. As a result, it plays no
role in the control optimality conditions examined below.

Since $\ell(u)=|u|$ is nonsmooth at the origin, its subdifferential is given by
\[
\partial |u| =
\begin{cases}
\{1\}, & u>0,\\
[-1,1], & u=0,\\
\{-1\}, & u<0.
\end{cases}
\]
Consequently, the subdifferentials of the plant and model Hamiltonians with
respect to the control are
\[
\partial_u H(u) = \partial |u| + a,
\qquad
\partial_u \hat H(u) = \partial |u| + \hat a.
\]

Suppose there exists $\bar u\in\mathcal U$ such that the subgradient matching
condition of Theorem~\ref{thm:equiv_optionA} holds, namely,
\begin{equation}
\partial_u \hat H(\bar u) = \partial_u H(\bar u),
\label{eq:ex1_subgrad_match}
\end{equation}
and that $\bar u$ also satisfies the normal-cone optimality condition
\begin{equation}
0 \in \partial_u \hat H(\bar u) + N_{\mathcal U}(\bar u).
\label{eq:ex1_normal_cone}
\end{equation}
Then, by Theorem~\ref{thm:equiv_optionA}, $\bar u$ minimizes both the model and
plant Hamiltonians over the admissible control set $\mathcal U$.

In particular, when $\bar u = 0$, the subgradient matching condition
\eqref{eq:ex1_subgrad_match} reduces to the requirement $a=\hat a$, in which case
both Hamiltonians admit $u=0$ as a minimizer. Since $0$ lies in the interior of
$\mathcal U$, the normal-cone condition \eqref{eq:ex1_normal_cone} is
automatically satisfied.

In this example, we highlight that equivalence between model-based and plant-based
optimal controls does not require smooth costs or identical dynamics. Instead,
it relies on alignment of the first-order optimality conditions, expressed here
through subgradient matching and feasibility with respect to the control
constraints, while the penalty term influences the state evolution but not the
instantaneous control minimization.

\subsubsection{Example 2}

Let $\mathcal U=\mathbb R^m$ (not compact) and
\[
\ell(t,z,u)=\ell_0(t,z)+\frac{1}{2}u^\top R(t)u,\qquad R(t)\succeq r_{\min}I.
\]
Consider further that the dependence of $f$ and $\hat f$ on $u$ is such that $u\mapsto \lambda^\top f(t,x,u)$ and $u\mapsto \hat\lambda^\top \hat f(t,\hat x,u)$ are convex and satisfy the linear growth bound in Assumption~\ref{ass:convex_growth_B}. Then each Hamiltonian is coercive and strongly convex in $u$, hence admits a unique minimizer despite $\mathcal U$ being unbounded (Lemma~\ref{lem:exist_unique_B}).

If, in addition, along the relevant trajectories the switching-gradient match \eqref{eq:struct_match_switching} holds (e.g., under \eqref{eq:sufficient_B} and state alignment), then the unique minimizers coincide and $u^\circ(\cdot)\equiv u^*(\cdot)$ almost everywhere, illustrating Theorem~\ref{thm:equiv_optionB}.

The results above characterize when optimal control policies derived from an
approximate model coincide with those of the actual system, both in general
convex settings and under additional structural assumptions that ensure
uniqueness. Next, we illustrate these theoretical findings through a numerical
example that highlights the role of model mismatch, penalty terms, and
Hamiltonian minimization in shaping the resulting control strategies.

\section{Numerical Example -- Equivalent Optimal Controls under Model Mismatch}

In this section, we present a numerical example that demonstrates the
equivalence results developed in Sections~IV.A–IV.C and illustrates how a
penalized model-based formulation can recover the plant-optimal control despite
significant model mismatch.

\subsection{Plant and model dynamics}

We consider scalar systems over a finite horizon $[0,T]$ with $T=6$.
The actual system (plant) is given by
\begin{equation}
\dot{\hat x}(t) = \hat a\,\hat x(t) + \hat b\,u(t),
\qquad \hat x(0)=1.5,
\label{eq:num_plant}
\end{equation}
with
\[
\hat a = 0.3, \qquad \hat b = 1.3.
\]

The available model (digital twin) is
\begin{equation}
\dot x(t) = a\,x(t) + b\,u(t),
\qquad x(0)=1.5,
\label{eq:num_model}
\end{equation}
with
\[
a = -0.6, \qquad b = 0.7.
\]

Thus, both the drift and control effectiveness differ between the plant and the
model, i.e., $a \neq \hat a$ and $b \neq \hat b$.

Next, we specify the performance objective and admissible control constraints used
to evaluate the resulting control strategies.

\subsection{Cost function and constraints}

The plant performance index is
\begin{equation}
J_{\mathrm{act}}(u) =
\int_0^T
\Big(
q\,\hat x(t)^2
+ \tfrac{1}{2} r\,u(t)^2
+ d(t)\,u(t)
\Big)\,dt
+
\tfrac{1}{2} q_T \hat x(T)^2,
\label{eq:num_cost}
\end{equation}
with
\[
q = 0.5, \qquad r = 0.2, \qquad q_T = 2.
\]

The time-varying term
\begin{equation}
d(t) = A\,\mathrm{sign}\!\big(\sin(\omega t)\big),
\qquad
A = 200, \quad \omega = \frac{4\pi}{T},
\label{eq:num_d}
\end{equation}
acts as a known exogenous excitation that induces switching in the optimal
control.

The admissible control set is the compact interval
\begin{equation}
U = [-0.05,\;0.05].
\label{eq:num_U}
\end{equation}

The model-based problem uses the same running and terminal costs, augmented with
a penalty term $\beta(t)\bigl(x(t)-\hat x(t)\bigr)^2$, which does not affect the
pointwise minimization with respect to the control input.

With the dynamics, cost function, and constraints defined, next, we derive the
corresponding Hamiltonians and characterize the resulting optimal control laws
for both the plant and the model-based problems.

\subsection{Hamiltonian minimization and control laws}

Next, we characterize the optimal control laws resulting from the Hamiltonian
minimization of both the actual system (plant) and the available model.

\paragraph{Plant Hamiltonian}
For the actual system with dynamics \eqref{eq:num_plant} and running cost
\[
L(t,\hat x,u) = q\,\hat x^2 + \frac{r}{2}u^2 + d(t)u,
\]
the Hamiltonian is
\[
\hat H(t,\hat x,u,\hat\lambda)
=
q\,\hat x^2 + \frac{r}{2}u^2 + d(t)u
+ \hat\lambda\bigl(\hat a\,\hat x + \hat b\,u\bigr).
\]
Rearranging terms yields
\[
\hat H(t,\hat x,u,\hat\lambda)
=
\frac{r}{2}u^2
+ \bigl(d(t) + \hat b\,\hat\lambda(t)\bigr)u
+ q\,\hat x(t)^2 + \hat a\,\hat\lambda(t)\,\hat x(t),
\]
where the last two terms are independent of the control input $u$.
Since $r>0$, the Hamiltonian is strictly convex in $u$, and the unconstrained
minimizer is obtained by setting $\partial \hat H/\partial u = 0$, yielding
\[
u_{\mathrm{p,uncon}}(t)
=
-\frac{d(t) + \hat b\,\hat\lambda(t)}{r}.
\]

\paragraph{Model Hamiltonian}
Next, we repeat this construction for the model-based problem, which incorporates
a penalty term to account for deviation from the plant trajectory.
The model-based optimal control problem uses the approximate dynamics
\eqref{eq:num_model} and a penalized running cost of the form
\[
L_{\mathrm{m}}(t,x,\hat x,u)
=
q\,x^2 + \frac{r}{2}u^2 + d(t)u
+ \beta(t)\bigl(x - \hat x\bigr)^2,
\]
where $\beta(t)\ge 0$ is a time-varying penalty weight.
The corresponding Hamiltonian is
\[
H(t,x,\hat x,u,\lambda)
=
q\,x^2 + \frac{r}{2}u^2 + d(t)u
+ \beta(t)\bigl(x - \hat x\bigr)^2
+ \lambda\bigl(a\,x + b\,u\bigr).
\]
Rewriting, we obtain
\[
H(t,x,\hat x,u,\lambda)
=
\frac{r}{2}u^2
+ \bigl(d(t) + b\,\lambda(t)\bigr)u
+ q\,x(t)^2
+ \beta(t)\bigl(x(t) - \hat x(t)\bigr)^2
+ a\,\lambda(t)\,x(t).
\]
Importantly, the penalty term $\beta(t)\bigl(x-\hat x\bigr)^2$ does not depend on
$u$ and therefore does not affect the pointwise minimization of the Hamiltonian
with respect to the control input. The unconstrained minimizer of the model
Hamiltonian is thus given by
\[
u_{\mathrm{m,uncon}}(t)
=
-\frac{d(t) + b\,\lambda(t)}{r}.
\]

\paragraph{Constrained control law and equivalence.}
In both cases, the constrained optimal control is obtained by solving, at each
time $t$, the pointwise Hamiltonian minimization problem
\[
u^*(t) \in \arg\min_{u \in U} H(t,\cdot,u,\cdot),
\]
where $U=[u_{\min},u_{\max}]$ is a compact, convex set.
Since both Hamiltonians are strictly convex in $u$, the constrained minimizer is
obtained by projecting the unconstrained minimizer onto $U$, namely,
\[
u^*(t)
=
\Pi_U\!\left(u_{\mathrm{uncon}}(t)\right)
=
\begin{cases}
u_{\min}, & u_{\mathrm{uncon}}(t) < u_{\min},\\[2mm]
u_{\mathrm{uncon}}(t), & u_{\min} \le u_{\mathrm{uncon}}(t) \le u_{\max},\\[2mm]
u_{\max}, & u_{\mathrm{uncon}}(t) > u_{\max}.
\end{cases}
\]

Although the plant and model dynamics differ, the penalty term shapes the state
and costate evolution of the model-based problem without altering the structure
of the Hamiltonian minimization with respect to $u$. As a result, whenever the
projected minimizers coincide, the optimal control trajectories derived from the
model and the plant are identical.

\subsection{Results and discussion}

The resulting control laws allow us to compare, in simulation, the trajectories
generated by the plant-optimal and model-based optimal controls. In particular,
we examine whether the penalized model-based formulation recovers the same
control trajectory as the plant, despite differences in the underlying dynamics.

Figure~\ref{fig:control_equiv_nonconst} shows the optimal control trajectories obtained from the plant Hamiltonian ($u^*(t)$) and the model-based Hamiltonian ($u^\circ(t)$).
The two curves overlap exactly and are non-constant, exhibiting bang--bang behavior with multiple switching times. This figure directly illustrates the conclusion of Theorem~3. Equivalence of optimal control trajectories follows from equivalence of the constrained Hamiltonian minimizers, not from equality of the dynamics.
Figure~\ref{fig:unconstrained_nonconst} shows the unconstrained minimizers $u_{\mathrm{uncon}}^*(t)$ and $u_{\mathrm{uncon}}^\circ(t)$ together with the control bounds and explains \emph{why} Figure~\ref{fig:control_equiv_nonconst} shows equivalence.
Although the gradients of the plant and model Hamiltonians are different (as reflected in the different unconstrained minimizers), the admissible control constraints dominate the pointwise minimization.
Both unconstrained minimizers lie far outside $\mathcal U$ with the same sign, and therefore their projections onto $\mathcal U$ coincide.

From a theoretical perspective, this figure highlights that
\[
\nabla_u \hat H \neq \nabla_u H
\quad \text{while} \quad
\arg\min_{u\in\mathcal U} \hat H = \arg\min_{u\in\mathcal U} H,
\]
which is precisely the mechanism underlying the equivalence results of Section~4.

Figure~\ref{fig:state_nonconst} shows the state trajectories of the plant and the model under the common optimal control. This confirms that equivalence of controls does \emph{not} imply equivalence of state trajectories when the dynamics differ.

\begin{figure}[t]
\centering
\includegraphics[width=0.75\linewidth]{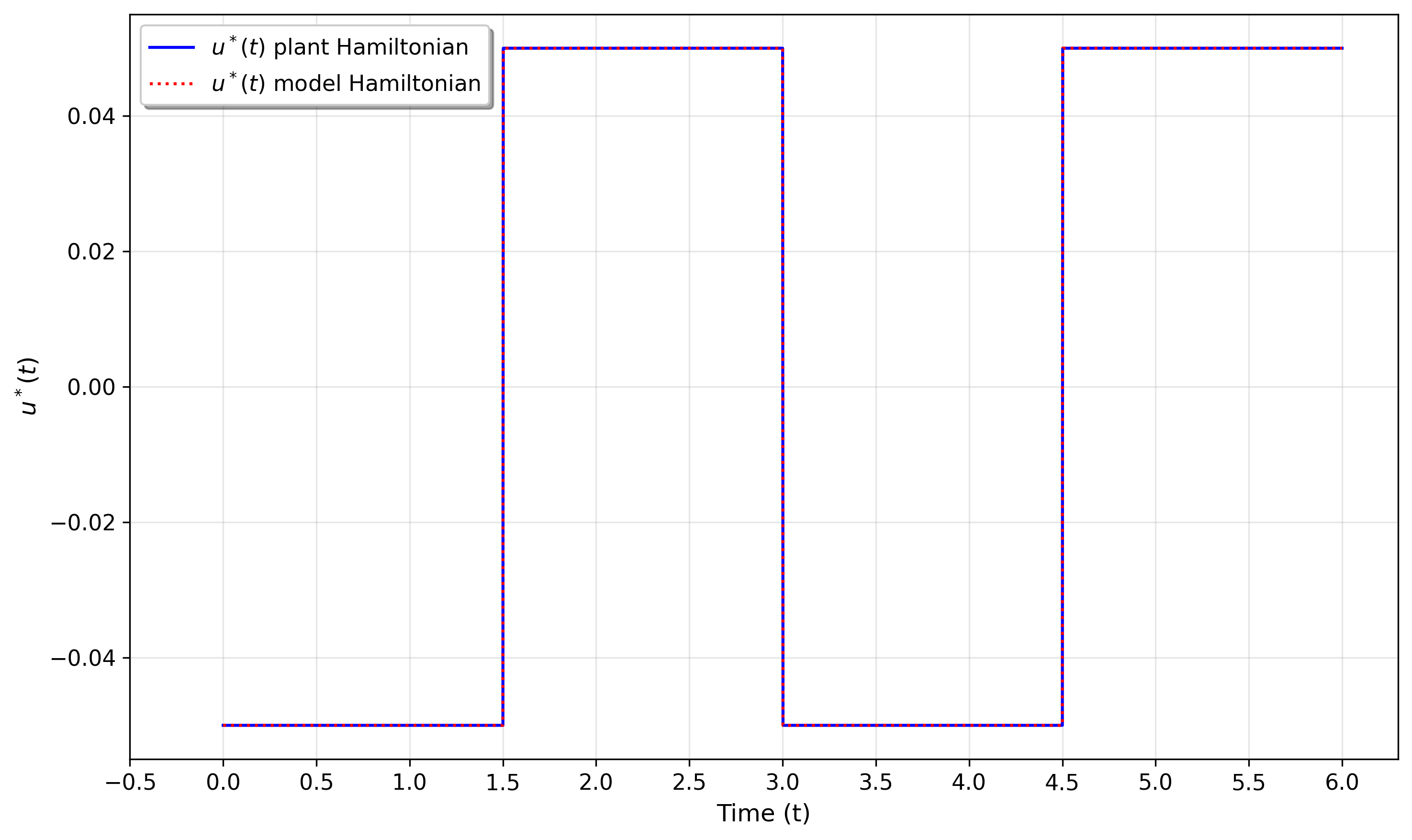}
\caption{Constrained optimal controls $u^*(t)$ (plant) and $u^\circ(t)$ (model). 
	This figure illustrates the equivalence of constrained Hamiltonian minimizers stated in Theorem 3.}
\label{fig:control_equiv_nonconst}
\end{figure}

\begin{figure}[t]
\centering
\includegraphics[width=0.75\linewidth]{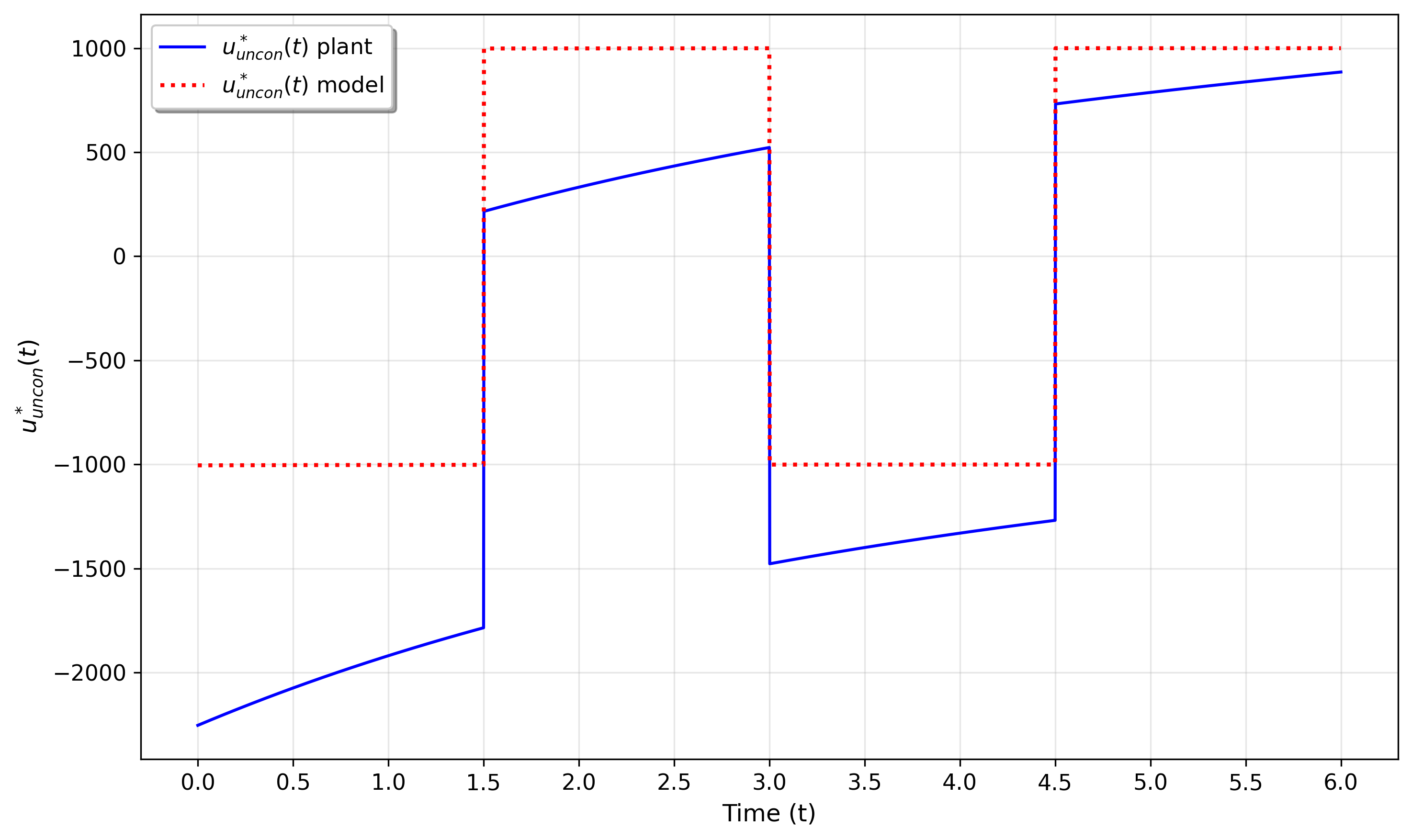}
\caption{Unconstrained Hamiltonian minimizers for the plant and model, together with control bounds. 
	The unconstrained minimizers differ due to model mismatch, but both lie far outside $\mathcal U$ for all $t$.}
\label{fig:unconstrained_nonconst}
\end{figure}

\begin{figure}[t]
\centering
\includegraphics[width=0.75\linewidth]{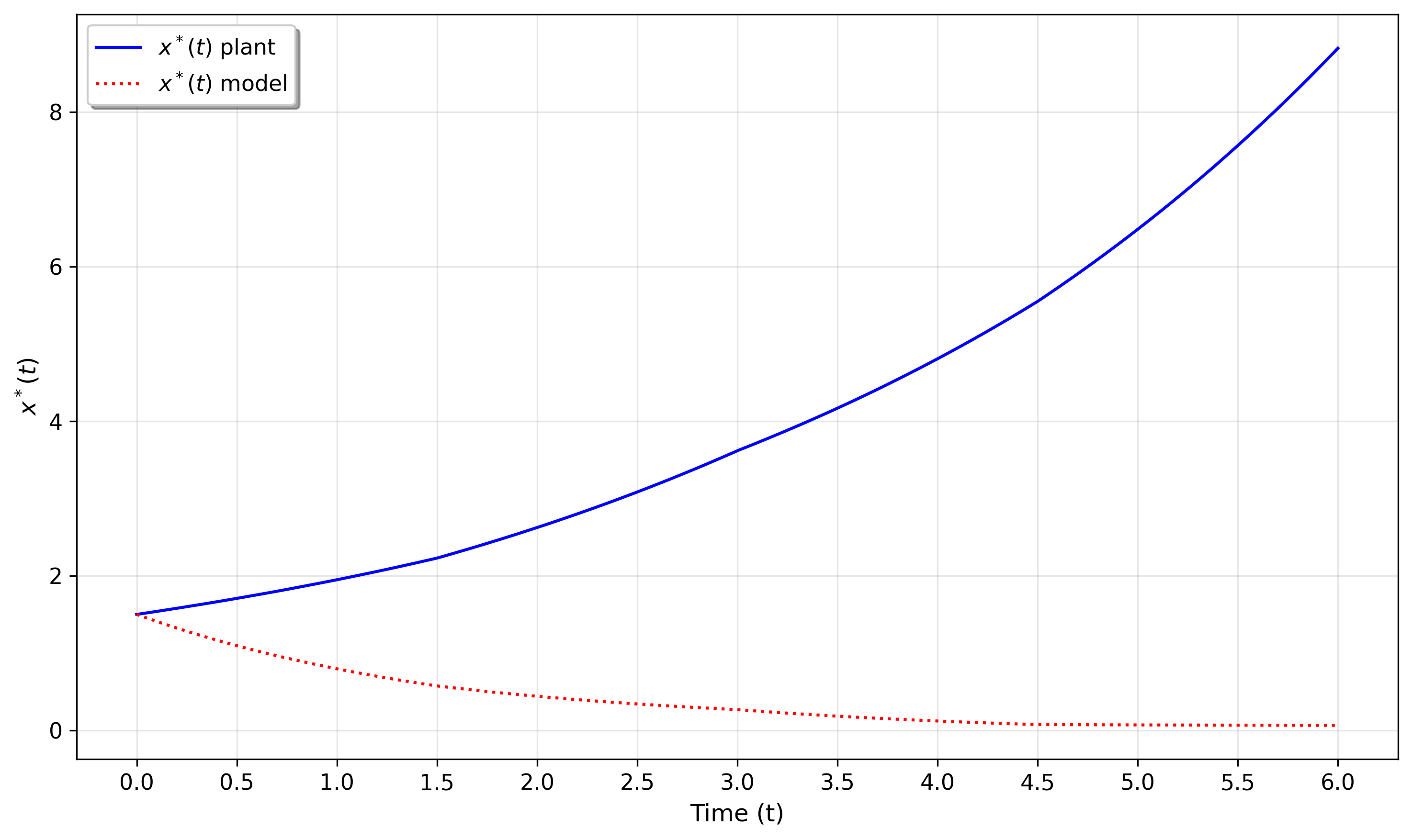}
\caption{State trajectories of the plant and the model under the identical optimal control.
	The trajectories differ due to mismatched dynamics, even though the control is the same.}
\label{fig:state_nonconst}
\end{figure}

This example confirms the theoretical results developed earlier. Even when the
model dynamics differ substantially from those of the plant, the inclusion of a
penalty term can preserve equivalence of the pointwise Hamiltonian minimization
and, consequently, the resulting optimal control. These findings underscore that
accurate control synthesis does not require an exact model, but rather alignment
of the optimality conditions that govern control selection.

From a learning perspective, this suggests that accurate control synthesis does not always require precise recovery of system dynamics. Instead, learning efforts can be focused on identifying regimes in which model-based Hamiltonian minimizers coincide with those of the plant. As a result, the framework provides a theoretical basis for robust model-based decision making in the presence of modeling error, and supports the use of digital twins as effective control surrogates even when they are not dynamically exact.

\section{Concluding Remarks} \label{sec:conclusion}
In this paper, we studied finite-horizon optimal control problems in which an approximate model is used to synthesize control strategies for an actual system with unknown dynamics. We introduced a penalized model-based formulation and analyzed the associated Hamiltonian systems to establish structural conditions under which the constrained Hamiltonian minimizers of the model-based and plant-based problems coincide. As a result, we showed that the optimal control trajectories derived from the available model are identical to those of the actual system, despite model mismatch.

A key insight of our analysis is that the penalty term capturing model–plant mismatch influences the state and costate evolution but does not enter explicitly in the pointwise minimization of the Hamiltonian with respect to the control input. This observation allows us to decouple questions of model accuracy from control optimality and provides a principled explanation for why approximate models and digital twins can successfully generate optimal control strategies in practice. The numerical examples further illustrate that the equivalence of optimal controls can hold even when state trajectories differ and when the optimal control is non-constant.

The results of this paper suggest a shift in perspective for learning-based control. Rather than focusing on exact system identification, learning efforts can be directed toward preserving the structural properties that determine Hamiltonian minimization. Ongoing work explores implementing this analysis in stochastic systems. A potential direction for future research should explore extensions to decentralized and partially observed systems, as well as the integration of online learning mechanisms for adapting the penalty structure in real time.

\section{Acknowledgments}
This research was supported in part by NSF under Grants CNS-2401007, CMMI-2348381, IIS-2415478, and in part by MathWorks.

The author would like to thank Xiaoxing Ren for helping in producing the plots.

\bibliographystyle{elsarticle-num}
\bibliography{TAC_learn_Andreas,IDS,TAC2_learn,TAC_Ref_Andreas,TAC_Ref_structure}

@article{Malikopoulos2024,
	author = {Malikopoulos, Andreas A.},
	journal = {European Journal of Control},
	number = {Part A},
	pages = {101043},
	title = {Combining Learning and Control in Linear Systems},
	volume = {80},
	year = {2024}}

@article{Malikopoulos2022a,
	author = {Malikopoulos, Andreas A},
	journal = {Automatica},
	number = {110912},
	title = {Separation of Learning and Control for Cyber-Physical Systems},
	volume = {151},
	year = {2023}}

@inproceedings{chalaki2020ICCA,
	arxivid = {1903.05252},
	author = {Chalaki, Behdad and Beaver, Logan E and Remer, Ben and Jang, Kathy and Vinitsky, Eugene and Bayen, Alexandre and Malikopoulos, Andreas A},
	booktitle = {IEEE 16th International Conference on Control \& Automation (ICCA)},
	pages = {35--40},
	title = {Zero-shot autonomous vehicle policy transfer: From simulation to real-world via adversarial learning},
	year = {2020}}

@inproceedings{jang2019simulation,
	author = {Jang, Kathy and Vinitsky, Eugene and Chalaki, Behdad and Remer, Ben and Beaver, Logan and Malikopoulos, Andreas A and Bayen, Alexandre},
	booktitle = {Proceedings of the 10th ACM/IEEE International Conference on Cyber-Physical Systems},
	pages = {291--300},
	title = {Simulation to scaled city: zero-shot policy transfer for traffic control via autonomous vehicles},
	year = {2019}}

@inproceedings{kounatidis2025combined,
	arxivid = {2510.00308},
	author = {Kounatidis, Panagiotis and Malikopoulos, Andreas A.},
	booktitle = {65th American Control Conference (ACC)},
	note = {in review},
	title = {Combined Learning and Control: A New Paradigm for Optimal Control with Unknown Dynamics},
	year = {2025 (arXiv: 2510.00308)}}

@article{Subramanian2019ApproximateIS,
	author = {Jayakumar Subramanian and Aditya Mahajan},
	journal = {2019 IEEE 58th Conference on Decision and Control (CDC)},
	pages = {1629-1636},
	title = {Approximate information state for partially observed systems},
	year = {2019}}

@article{Subramanian2020ApproximateIS,
	author = {Jayakumar Subramanian and Amit Sinha and Raihan Seraj and A. Mahajan},
	journal = {Journal of Machine Learning Research},
	pages = {1--83},
	title = {Approximate information state for approximate planning and reinforcement learning in partially observed systems},
	volume = {23},
	year = {2022}}

@inproceedings{Kara:2018vu,
	author = {A. D. Kara and S. Y{\"u}ksel},
	booktitle = {2018 IEEE Conference on Decision and Control (CDC)},
	doi = {10.1109/CDC.2018.8619684},
	isbn = {2576-2370},
	journal = {2018 IEEE Conference on Decision and Control (CDC)},
	journal1 = {2018 IEEE Conference on Decision and Control (CDC)},
	pages = {2753--2758},
	title = {Robustness to Incorrect System Models in Stochastic Control and Application to Data-Driven Learning},
	ty = {CONF},
	year = {2018},
	year1 = {17-19 Dec. 2018}}

@book{Sutton1998a,
	author = {Sutton,R. S. and Barto,A. G.},
	publisher = {Bradford Books},
	title = {Reinforcement Learning: An Introduction},
	year = {1998}}

@book{Bertsekas1996,
	author = {Bertsekas,D. P. and Tsitsiklis,J. N.},
	publisher = {Athena Scientific},
	title = {Neuro-Dynamic Programming},
	year = {1996}}

@article{Malikopoulos2009,
	author = {Malikopoulos, Andreas A and Papalambros, Panos Y. and Assanis, Dennis N},
	journal = {Journal of Dynamic Systems, Measurement and Control, Transactions of the ASME},
	number = {4},
	title = {{A real-time computational learning model for sequential decision-making problems under uncertainty}},
	volume = {131},
	year = {2009}}

@article{Malikopoulos2009b,
	author = {Malikopoulos, Andreas A.},
	journal = {J. Dyn. Sys., Meas., Control},
	number = {4},
	pages = {041011--7},
	title = {{Convergence properties of a computational learning model for unknown Markov chains}},
	volume = {131},
	year = {2009}}

@article{Malikopoulos2010a,
	author = {Malikopoulos, Andreas A. and Papalambros, Panos Y. and Assanis, Dennis N.},
	journal = {Journal of Dynamic Systems, Measurement, and Control},
	number = {2},
	pages = {024504--024504},
	title = {Online Identification and Stochastic Control for Autonomous Internal Combustion Engines},
	volume = {132},
	year = {2010}}

@article{Arslan:2017vo,
	author = {G. Arslan and S. Y{\"u}ksel},
	journal = {IEEE Transactions on Automatic Control},
	number = {4},
	pages = {1545--1558},
	title = {Decentralized Q-Learning for Stochastic Teams and Games},
	volume = {62},
	year = {2017}}

@article{Kiumarsi:2018tq,
	author = {B. Kiumarsi and K. G. Vamvoudakis and H. Modares and F. L. Lewis},
	journal = {IEEE Transactions on Neural Networks and Learning Systems},
	number = {6},
	pages = {2042--2062},
	title = {Optimal and Autonomous Control Using Reinforcement Learning: A Survey},
	volume = {29},
	year = {2018}}

@article{You:2019va,
	author = {You, Changxi and Lu, Jianbo and Filev, Dimitar and Tsiotras, Panagiotis},
	journal = {Robotics and Autonomous Systems},
	pages = {1--18},
	title = {Advanced planning for autonomous vehicles using reinforcement learning and deep inverse reinforcement learning},
	volume = {114},
	year = {2019}}

@article{Sahoo:2020tx,
	author = {P. P. Sahoo and K. G. Vamvoudakis},
	journal = {IEEE Control Systems Letters},
	number = {3},
	pages = {749--754},
	title = {On-Off Adversarially Robust Q-Learning},
	volume = {4},
	year = {2020}}

@article{Gatsis:2021tt,
	author = {Gatsis, Konstantinos and Pappas, George J.},
	journal = {Automatica},
	pages = {109386},
	title = {Statistical learning for analysis of networked control systems over unknown channels},
	volume = {125},
	year = {2021}}

@article{Liu:2014wb,
	author = {N. Liu and A. G. Alleyne},
	journal = {IEEE Transactions on Control Systems Technology},
	number = {1},
	pages = {331--337},
	title = {Iterative Learning Identification Applied to Automated Off-Highway Vehicle},
	volume = {22},
	year = {2014}}

@article{Armstrong:2021vw,
	author = {A. A. Armstrong and A. J. Wagoner Johnson and A. G. Alleyne},
	booktitle = {IEEE Transactions on Control Systems Technology},
	journal = {IEEE Transactions on Control Systems Technology},
	journal1 = {IEEE Transactions on Control Systems Technology},
	number = {2},
	pages = {546--555},
	title = {An Improved Approach to Iterative Learning Control for Uncertain Systems},
	ty = {JOUR},
	vo = {29},
	volume = {29},
	year = {2021}}

@article{Khong:2016ws,
	author = {Khong, Sei Zhen and Ne{\v s}i{\'c}, Dragan and Krsti{\'c}, Miroslav},
	journal = {IFAC-PapersOnLine},
	number = {18},
	pages = {962--967},
	title = {An extremum seeking approach to sampled-data iterative learning control of continuous-time nonlinear systems},
	volume = {49},
	year = {2016}}

@article{Khong:2016ug,
	author = {Khong, Sei Zhen and Ne{\v s}i{\'c}, Dragan and Krsti{\'c}, Miroslav},
	journal = {Automatica},
	pages = {238--245},
	title = {Iterative learning control based on extremum seeking},
	volume = {66},
	year = {2016}}

@inproceedings{Vinitsky:2018vx,
	author = {E. Vinitsky and K. Parvate and A. Kreidieh and C. Wu and A. Bayen},
	booktitle = {2018 21st International Conference on Intelligent Transportation Systems (ITSC)},
	journal = {2018 21st International Conference on Intelligent Transportation Systems (ITSC)},
	journal1 = {2018 21st International Conference on Intelligent Transportation Systems (ITSC)},
	pages = {759--765},
	title = {Lagrangian Control through Deep-RL: Applications to Bottleneck Decongestion},
	ty = {CONF},
	year = {2018},
	year1 = {4-7 Nov. 2018}}

@article{Krichene:2015vx,
	author = {Krichene, Walid and Drigh{\`e}s, Benjamin and Bayen, Alexandre M.},
	booktitle = {SIAM Journal on Control and Optimization},
	isbn = {0363-0129},
	journal = {SIAM Journal on Control and Optimization},
	journal1 = {SIAM J. Control Optim.},
	number = {2},
	pages = {1056--1081},
	publisher = {Society for Industrial and Applied Mathematics},
	title = {Online Learning of Nash Equilibria in Congestion Games},
	ty = {JOUR},
	volume = {53},
	year = {2015},
	year1 = {2015}}

@article{Zhang:2020wf,
	author = {X. Zhang and M. Bujarbaruah and F. Borrelli},
	booktitle = {IEEE Transactions on Control Systems Technology},
	isbn = {1558-0865},
	journal = {IEEE Transactions on Control Systems Technology},
	journal1 = {IEEE Transactions on Control Systems Technology},
	pages = {1--13},
	title = {Near-Optimal Rapid MPC Using Neural Networks: A Primal-Dual Policy Learning Framework},
	ty = {JOUR},
	year = {2020}}

@article{Rosolia:2018wv,
	author = {U. Rosolia and F. Borrelli},
	booktitle = {IEEE Transactions on Automatic Control},
	isbn = {1558-2523},
	journal = {IEEE Transactions on Automatic Control},
	journal1 = {IEEE Transactions on Automatic Control},
	number = {7},
	pages = {1883--1896},
	title = {Learning Model Predictive Control for Iterative Tasks. A Data-Driven Control Framework},
	ty = {JOUR},
	vo = {63},
	volume = {63},
	year = {2018},
	year1 = {July 2018}}

@article{Rosolia:2020uo,
	author = {U. Rosolia and F. Borrelli},
	journal = {IEEE Transactions on Control Systems Technology},
	number = {6},
	pages = {2713--2719},
	title = {Learning How to Autonomously Race a Car: A Predictive Control Approach},
	volume = {28},
	year = {2020}}

@article{Fisac:2019wb,
	author = {J. F. Fisac and A. K. Akametalu and M. N. Zeilinger and S. Kaynama and J. Gillula and C. J. Tomlin},
	journal = {IEEE Transactions on Automatic Control},
	number = {7},
	pages = {2737--2752},
	title = {A General Safety Framework for Learning-Based Control in Uncertain Robotic Systems},
	volume = {64},
	year = {2019}}

@inproceedings{Akametalu:2014th,
	author = {A. K. Akametalu and J. F. Fisac and J. H. Gillula and S. Kaynama and M. N. Zeilinger and C. J. Tomlin},
	booktitle = {53rd IEEE Conference on Decision and Control},
	pages = {1424--1431},
	title = {Reachability-based safe learning with Gaussian processes},
	year = {2014}}

@article{Sutton:1992ub,
	author = {R. S. Sutton and A. G. Barto and R. J. Williams},
	journal = {IEEE Control Systems Magazine},
	number = {2},
	pages = {19--22},
	title = {Reinforcement learning is direct adaptive optimal control},
	volume = {12},
	year = {1992}}

@article{Recht2018ATO,
	author = {B. Recht},
	journal = {Annual Review of Control, Robotics, and Autonomous Systems},
	pages = {253-279},
	title = {A Tour of Reinforcement Learning: The View from Continuous Control},
	volume = {2},
	year = {2019}}

@article{Dydek2013AdaptiveCO,
	author = {Z. Dydek and A. Annaswamy and E. Lavretsky},
	journal = {IEEE Transactions on Control Systems Technology},
	pages = {1400-1406},
	title = {Adaptive Control of Quadrotor UAVs: A Design Trade Study With Flight Evaluations},
	volume = {21},
	year = {2013}}

@inproceedings{Leman2009L1AC,
	author = {Tyler Leman and E. Xargay and G. Dullerud and N. Hovakimyan and T. Wendel},
	booktitle = {AIAA guidance, navigation, and control conference},
	title = {L1 adaptive control augmentation system for the X-48B Aircraft},
	year = {2009}}

@book{Ioannou2012RobustAC,
	author = {Petros A. Ioannou and Jing Sun},
	publisher = {PTR Prentice-Hall},
	title = {Robust Adaptive Control},
	year = {1996}}

@book{strm1989AdaptiveC,
	author = {K. {\AA}str{\"o}m and B. Wittenmark},
	publisher = {Addison-Wesley Publising Company},
	title = {Adaptive Control},
	year = {1995}}

@article{Guha2021OnlinePF,
	author = {A. Guha and A. Annaswamy},
	journal = {ArXiv},
	title = {Online Policies for Real-Time Control Using MRAC-RL},
	volume = {abs/2103.16551},
	year = {2021}}







\end{document}